\documentclass{amsart}
\usepackage{amsfonts}
\newtheorem{thm}{Theorem}[section]

\newtheorem{lem}[thm]{Lemma}

\newtheorem{defn}[thm]{Definition}

\def\ll{{\mathcal L}}
\begin{document}
\title[On complex nilpotent Leibniz superalgebras of nilindex {\rm n+m}]
{On complex nilpotent Leibniz superalgebras \\ of nilindex {\rm
n+m}}
\author{L.M. Camacho, J. R. G\'{o}mez, B.A. Omirov and A.Kh. Khudoyberdiyev}
\address{[L.M. Camacho -- J.R. G\'{o}mez] Dpto. Matem\'{a}tica Aplicada I.
Universidad de Sevilla. Avda. Reina Mercedes, s/n. 41012 Sevilla.
(Spain)} \email{lcamacho@us.es --- jrgomez@us.es}
\address{[B.A. Omirov -- A.Kh. Khudoyberdiyev] Institute of Mathematics and Information Technologies
 of Academy of Uzbekistan, 29, F.Hodjaev srt., 100125, Tashkent (Uzbekistan)}
\email{omirovb@mail.ru --- khabror@mail.ru}


\thanks{Partially supported by the PAICYT, FQM143 of the Junta de Andaluc\'{\i}a (Spain).
The third author was supported by grant NATO-Reintegration ref.
CBP.EAP.RIG.983169}

\begin{abstract}
We present the description up to isomorphism of Leibniz
superalgebras with characteristic sequence $(n|m_1,\dots,m_k)$ and
nilindex $n+m,$ where $m=m_1+\cdots+m_k,$ $n$ and $m$ ($m\neq 0$)
are dimensions of even and odd parts, respectively.
\end{abstract} \maketitle



\section{Introduction}

Extensive investigations of Lie algebras theory have lead to
appearance of more general algebraic objects, such as Mal'cev
algebras, Lie superalgebras, Leibniz algebras, and others.

The well-known Lie superalgebras are generalizations of Lie
algebras and for many years they attract the attention of both the
mathematicians and physicists. The systematical exposition of
basic Lie superalgebras theory can be found in the monograph
\cite{K} and the papers related with the nilpotent Lie
superalgebras are \cite{2007Yu}, \cite{GL1}, \cite{G-K-N}, and
\cite{G-K-N1}.

Leibniz superalgebras are the generalization of Leibniz algebras
and, on the other hand, they naturally generalize Lie
superalgebras.

Recall that Leibniz algebras are a "non antisymmetric"
generalization of Lie algebras \cite{loday}. The study of
nilpotent Leibniz algebras \cite{Alb}--\cite{Alb-Omirov-Rak} shows
that many nilpotent properties of Lie algebras can be extended for
nilpotent Leibniz algebras. The results of the nilpotent Leibniz
algebras may help us to investigate the nilpotent Leibniz
superalgebras.

In the description of Leibniz superalgebras structure the crucial
task is to prove  the existence of a suitable basis (so-called the
adapted basis) in which the multiplication of the superalgebra has
the most convenient form.

In contrast to Lie superalgebras for which problem of the
description of superalgebras with the maximal nilindex is
difficult \cite{G-K-N}, for nilpotent Leibniz superalgebras it
turns out to be comparatively easy and was solved in \cite{Alb}.
The distinctive property of such Leibniz superalgebras is that
they are single-generated. The next step - the description of
Leibniz superalgebras with the dimensions of even and odd parts,
respectively equal to $n$ and $m,$ and with nilindex $n+m$ at this
moment seems to be very complicated. Therefore, such Leibniz
superalgebras can be studied by applying restrictions on their
characteristic sequences \cite{FilSup}--\cite{G-N-O}. Following
this approach we investigate such Leibniz superalgebras with
characteristic sequence $C(L)=(n\ |\ m_1,\dots, m_k)$, where
$m_1+\dots+m_k=m.$

Taking into account the results of the papers \cite{C-G-N-O} and
\cite{G-N-O}, where some cases of nilpotent Leibniz superalgebras
were described, in this work we investigate the rest cases. Thus,
we complete the description of Leibniz superalgebras with
characteristic sequence $C(L)=(n\ |\ m_1,\dots, m_k)$ and nilindex
$n+m.$

All over the work we consider spaces and algebras over the field
of complex numbers. By asterisks $(*)$ we denote the appropriate
coefficients at the basic elements of superalgebra.

\section{Preliminaries}

A $\mathbb{Z}_2$-graded vector space
$\mathcal{G}=\mathcal{G}_0\oplus \mathcal{G}_1$ is called {\it a
Lie superalgebra} if it is equipped with a product $[-,-]$ which
satisfies the following conditions:

1. $[\mathcal{G}_\alpha,\mathcal{G}_\beta]\subseteq
\mathcal{G}_{\alpha+\beta(mod\ 2)}$,

2. $[x,y]=-(-1)^{\alpha\beta}[y,x],$

3. $(-1)^{\alpha\gamma}
[x,[y,z]]+(-1)^{\alpha\beta}[y,[z,x]]+(-1)^{\beta\gamma}[z,[x,y]]=0$
-- {\it Jacobi superidentity}\\
for any $x\in \mathcal{G}_\alpha,$ $y\in \mathcal{G}_\beta,$ $z\in
\mathcal{G}_\gamma$ and $\alpha,\beta, \gamma\in \mathbb{Z}_2.$

A $\mathbb{Z}_2$-graded vector space
$\mathcal{L}=\mathcal{L}_0\oplus \mathcal{L}_1$ is called {\it a
Leibniz superalgebra} if it is equipped with a product $[-, -]$
which satisfies the following conditions:

1. $[\mathcal{L}_\alpha,\mathcal{L}_\beta]\subseteq
\mathcal{L}_{\alpha+\beta(mod\ 2)},$

2. $[x, [y, z]]=[[x, y], z] - (-1)^{\alpha\beta} [[x, z], y] -$
  {\it Leibniz superidentity},\\
 for any $x\in \mathcal{L},$ $y \in \mathcal{L}_\alpha,$ $z \in
 \mathcal{L}_\beta$ and $\alpha,\beta\in \mathbb{Z}_2.$

The vector spaces $\mathcal{L}_0$ and $\mathcal{L}_1$ are said to
be even and odd parts of the superalgebra $\mathcal{L}$,
respectively.

Note that if in $\mathcal{L}$ the graded identity
$[x,y]=-(-1)^{\alpha\beta} [y,x]$ holds, then the Leibniz
superidentity and  Jacobi superidentity coincide. Thus, Leibniz
superalgebras are a generalization of Lie superalgebras.

For examples of Leibniz superalgebras we refer to \cite{Alb}.

Denote by $Leib^{n,m}$ the set of Leibniz superalgebras with
dimensions of the even part and the odd part equal to $n$ and $m,$
respectively.

Let $V=V_0\oplus V_1,$ $W=W_0\oplus W_1$ be two
$\mathbb{Z}_2$-graded spaces. A linear map $f : V \to W$ is called
of degree $\alpha$ (denoted as $deg(f)=\alpha$), if
$f(V_\beta)\subseteq W_{\alpha +\beta}$ for all $\beta\in
\mathbb{Z}_2.$

Let $\mathcal{L}$ and $\mathcal{L}'$ be Leibniz superalgebras. A
linear map $f : \mathcal{L}\to \mathcal{L}'$ is called \emph{a
homomorphism} of Leibniz superalgebras if

1. $f(\mathcal{L}_0)\subseteq \mathcal{L}_0'$ and
$f(\mathcal{L}_1)\subseteq \mathcal{L}_1',$ i.e. $deg(f)=0$;

2. $f([x,y])=[f(x),f(y)]$ for all $x,y\in \mathcal{L}.$\\
Moreover, if  $f$ is bijection then it is called \emph{an
isomorphism} of Leibniz superalgebras $\mathcal{L}$ and
$\mathcal{L}'.$

For a given Leibniz superalgebra $\mathcal{L}$ we define a
descending central sequence in the following way:
$$
\mathcal{L}^1=\mathcal{L},\quad
\mathcal{L}^{k+1}=[\mathcal{L}^k,\mathcal{L}^1], \quad k \geq 1.
$$

\begin{defn} A Leibniz superalgebra $\mathcal{L}$ is called
nilpotent, if there exists  $s\in\mathbb N$ such that
$\mathcal{L}^s=0.$ The minimal number $s$ with this property is
called nilindex of the superalgebra $\mathcal{L}.$
\end{defn}

The sets $$\mathcal{R(L)}=\left\{ z\in \mathcal{L}\ |\
[\mathcal{L}, z]=0\right\} \mbox{and} \ \mathcal{Z(L)}=\left\{
z\in \mathcal{L}\ |\ [\mathcal{L}, z]=[z, \mathcal{L}]=0\right\}$$
are called {\it the right annihilator} and {\it the center of a
superalgebra} $\mathcal{L},$ respectively.

Using the Leibniz superidentity it is not difficult to see that
$\mathcal{R(L)}$ is an ideal of the superalgebra $\ll$. Moreover,
the elements of the form $[a,b]+(-1)^{\alpha \beta}[b,a],$ ($a \in
\ll_{\alpha}, \ b \in \ll_{\beta}$) belong to $\mathcal{R}(\ll)$.

The following theorem describes the nilpotent Leibniz
superalgebras with maximal nilindex.

\begin{thm} \label{t1} \cite{Alb} Let $\mathcal{L}$  be a Leibniz superalgebra of the variety
$Leib^{n,m }$ with nilindex equal to $n+m+1.$ Then $\mathcal{L}$
is isomorphic to one of the following non-isomorphic
superalgebras:
$$
[e_i,e_1]=e_{i+1},\ 1\le i\le n-1;\quad \left\{ \begin{array}{ll} [e_i,e_1]=e_{i+1},& 1\le i\le n+m-1, \\
{[}e_i,e_2{]}=2e_{i+2}, & 1\le i\le n+m-2,\\ \end{array}\right.
$$
(omitted products are equal to zero).
\end{thm}

It should be noted that for the second superalgebra we have $m=n$
when $n+m$ is even and $m=n+1$ if $n+m$ is odd. Moreover, it is
clear that the Leibniz superalgebra has the maximal nilindex if
and only if it is single-generated.

Let $\mathcal{L}=\mathcal{L}_0\oplus \mathcal{L}_1$ be a nilpotent
Leibniz superalgebra. For an arbitrary element $x\in L_0,$ the
operator of right multiplication $R_x$ is a nilpotent endomorphism
of the space $\mathcal{L}_i,$ where $i\in \{0, 1\}.$ Denote by
$C_i(x)$ ($i\in \{0, 1\}$) the descending sequence of the
dimensions of Jordan blocks of the operator $R_x.$ Consider the
lexicographical order on the set $C_i(\mathcal{L}_0)$.

\begin{defn} A sequence
$$
C(\mathcal{L})=\left( \left.\max\limits_{x\in L_0\setminus
[\mathcal{L}_0,\mathcal{L}_0]} C_0(x)\ \right|\
\max\limits_{\widetilde x\in \mathcal{L}_0\setminus
[\mathcal{L}_0,\mathcal{L}_0]} C_1\left(\widetilde x\right)
\right)
$$
is said to be the characteristic sequence of the Leibniz
superalgebra $\mathcal{L}.$
\end{defn}

Similar as in case of Lie superalgebras \cite{GL} (Corollary
3.0.1) it can be proved that the characteristic sequence is an
invariant under isomorphisms.

For Leibniz superalgebras we introduce the analogue of the
zero-filiform Leibniz algebras.
\begin{defn}
A Leibniz superalgebra $\ll\in$ $Leib^{n,m}$ is called {\it
zero-filiform} if \ $C(\ll)=(n|m)$
\end{defn}

Denote by $\mathcal{ZF}^{n,m}$ the set of all zero-filiform
Leibniz superalgebras from $Leib^{n,m}.$

From \cite{Omirov2} it can be concluded that the even part of a
zero-filiform Leibniz superalgebra is a zero-filiform Leibniz
algebra, therefore a zero-filiform superalgebra is not a Lie
superalgebra.

Further, we need the result on existence of an adapted basis for
zero-filiform Leibniz superalgebras.

\begin{thm}\cite{G-N-O} \label{t29}
In an arbitrary superalgebra from $\mathcal{ZF}^{n,m}$ there
exists a basis $\{x_1,x_2,\dots,x_n,y_1,y_2,\dots,y_m\}$ which
satisfies the following conditions:
$$\left\{\begin{array}{ll}
[x_i,x_1]=x_{i+1},&1\leq i \leq n-1,\\{} [x_n,x_1]=0,&\\{}
[x_i,x_k]=0,&1\leq i\leq n,\ 2\leq k\leq n,\\{} [y_j,x_1]=y_{j+1},
&1\leq j\leq m-1,\\{} [y_m,x_1]=0,&\\{} [y_j,x_k]=0,& 1\leq j\leq
m,\ 2\leq k\leq n.
\end{array}\right.$$
\end{thm}

\section{Zero-filiform Leibniz superalgebra with nilindex equal to n+m}

This section is devoted to the description of zero-filiform
Leibniz superalgebras with nilindex equal to $n+m$.

 Let $\ll \in \mathcal{ZF}^{n,m}$ with nilindex equal to $n+m.$
 Evidently, $\ll$ has two generators. Moreover, from Theorem \ref{t29} it
 follows that one generator lies in $\ll_0$ and the second generator lies in
 $\ll_1$. Without loss of generality it can be assumed that in an
 adapted basis the generators are $x_1$ and $y_1.$

In the adapted basis of $\ll$ we introduce the notations.
$$[x_i,y_1]=\sum\limits_{j=2}^m \alpha_{i,j} y_j, \ 1 \leq i \leq
n, \ \ \  [y_i,y_1]=\sum\limits_{j=2}^n \beta_{i,j} x_j, \ 1 \leq
i \leq m.$$

In the above notation the following lemma holds.

\begin{lem}\label{family}
$$[y_i,y_j]=\sum_{s=0}^{min\{i+j-1,m\}-i}(-1)^s C_{j-1}^{s}\sum_{t=2}^{n-j+s+1}\beta_{i+s,t}
x_{t+j-s-1}, \eqno(1)$$ where $1\leq i,j\leq m$.
\end{lem} \begin{proof}
The proof is deduced by induction on $j$ at any value of $i$.
\end{proof}
Since in the work \cite{G-N-O} the set $\mathcal{ZF}^{n,2}$ was
already described, we consider the set $\mathcal{ZF}^{n,m} (m\geq
3).$

\

{\bf{Case}} $\mathcal {ZF}^{2,m} \ (m\geq 3).$

\

\begin{thm}\label{thm32}
Let $\ll$ be a Leibniz superalgebra with the nilindex $m+2$ from
$\mathcal{ZF}^{2,m} \ (m\geq 3).$ Then $m$ is odd and $\ll$ is
isomorphic to the following superalgebra:
$$\begin{array}{ll}
[x_1,x_1]=x_2,&\\{} [y_i,x_1]=y_{i+1},& 1\leq i\leq m-1,\\{}
[x_1,y_i]=-y_{i+1},&1\leq i\leq m-1,\\{}
[y_i,y_{m+1-i}]=(-1)^{j+1}x_2, & 1\leq i\leq m-1.
\end{array}$$
\end{thm}
\begin{proof}

From (1) we easily obtain
$$\begin{array}{ll}
[y_i,y_j]=(-1)^{j-1}\beta_{i+j-1,2}x_2,&2\leq i+j\leq m+1,\\{}
[y_i,y_j]=0,&m+2\leq i+j\leq 2m.
\end{array}\eqno(2)$$

It should be noted that $\beta_{m,2}\neq 0.$ Indeed, if
$\beta_{m,2}=0,$ then $\ll^{m-1}=\{x_2, y_{m-1}, y_m\},$
$\ll^m=\{x_2, y_m\}$ and $\ll^{m+1}=\{y_m\}$ which imply that
$[y_{m-1},y_1]=ax_2$ and $[x_2,y_1]=by_m,$ where $ab\neq 0.$

The chain of the equalities
$$aby_m=[ax_2,y_1]=[[y_{m-1},y_1],y_1]=\frac{1}{2}[y_{m-1},[y_1,y_1]]=0
$$ implies a contradiction to the property $ab\neq 0.$ Therefore,
$\beta_{m,2}\neq 0.$

The simple analysis of the products leads to $x_2\in
\mathcal{Z(\ll)}$ (since $x_2\in\ll^{m+1} \subseteq
\mathcal{Z(\ll)}$).

Using the Leibniz superidentity we have

$$[x_1,y_i]=\alpha_{1,2}y_{i+1}+\cdots+\alpha_{1,m-i+1}y_m,\quad 1\leq i\leq m-1.$$

The expression $[y_1,x_1]+[x_1,y_1]$ lies in $\mathcal{R}(\ll).$
Hence
$$(1+\alpha_{1,2})y_2+\alpha_{1,3}y_3+\cdots+\alpha_{1,m}y_m \eqno(3)$$
belongs to $\mathcal{R}(\ll),$ as well.

If either $\alpha_{1,2}\neq -1$ or there exists $i \ (3 \leq i
\leq m)$ such that $\alpha_{1,i}\neq 0,$ then multiplying the
linear combination (3) from the right side required times to $x_1$
we deduce $y_m\in \mathcal{R}(\ll).$ However, by (2) we have
$[y_1,y_m]=(-1)^{m-1}\beta_{m,2}x_2$ which implies that $
\beta_{m,2}=0$ and we get a contradiction with condition
$\beta_{m,2}\neq 0.$

If $\alpha_{1,2}=-1$ and $\alpha_{1,i}=0\ (3\leq i\leq m)$, then
by applying the Leibniz superidentity for the basic elements
$\{x_1, y_i, y_i\}$ we obtain $\beta_{2i,2}=0$ for $1\leq i\leq
[\frac{m}{2}].$

Note that in case $m$ is even we obtain $\beta_{m,2}=0$ which is a
contradiction. Therefore, $m$ is odd.

Let us introduce new notations

$$\gamma_s=\beta_{2s-1,2}, \ 1\leq s \leq\frac{m+1}{2}.$$

Then we obtain the family
$L(\gamma_{1},\gamma_{2},\dots,\gamma_{\frac{m+1}{2}}):$
$$\begin{array}{l} \left\{
\begin{array}{ll}
[x_1,x_1]=x_2,&\\{} [y_i,x_1]=y_{i+1},&1\leq i\leq m-1,\\{}
[x_1,y_i]=-y_{i+1},&1\leq i\leq m-1,\\{}
[y_i,y_j]=(-1)^{j-1}\gamma_{\frac{i+j}{2},2}x_2,&i+j\   {\rm{ is \
even}}, \ 2\leq i+j\leq m+1,\ m \ \rm{ is \ odd}.
\end{array}\right.
\end{array}$$

Make the following general transformation of the generator basic
elements:
$$x'_1=b_1x_1, \ \ \ \ y'_1=\displaystyle\sum_{s=1}^{\frac{m+1}{2}}a_{2s-1}
y_{2s-1}.$$ Then $x'_2=b_1^2x_2$ and
$$y'_{2i-1}=b_1^{2(i-1)}\displaystyle\sum_{s=1}^{\frac{m-2(i-1)+1}{2}}a_{2s-1}
y_{2s+2i-3}, 1\leq i\leq \frac{m+1}{2},$$
$$
y'_{2i}=b_1^{2i-1}\displaystyle\sum_{s=1}^{\frac{m-2(i-1)-1}{2}}a_{2s-1}
y_{2s+2i-2}, 1\leq i\leq \frac{m-1}{2}. $$

Choosing the parameters $a_i$ as follows
$$a_1=\sqrt{\displaystyle\frac{1}{b_1^{m-3}\gamma_{\frac{m+1}{2}}}}, \
\ \  a_3=-\displaystyle\frac{a_1\gamma_{\frac{m-1}{2}}}
{2\gamma_{\frac{m+1}{2}}},$$
$$
a_i=-\displaystyle\frac{a_1^2\gamma_{\frac{m-i+2}{2}}+2a_1a_3\gamma_{\frac{m-i+4}{2}}+\cdots
+(2a_1a_{i-2}+\cdots+2a_{\frac{i-3}{2}}a_{\frac{i+1}{2}})
\gamma_{\frac{m-1}{2}}}{2a_1\gamma_{\frac{m+1}{2}}}-$$
$$-\displaystyle\frac{(2a_3a_{i-2}+\dots+2a_{\frac{i-3}{2}}a_{\frac{i+5}{2}}+a^2_{\frac{i+1}{2}})
\gamma_{\frac{m+1}{2}}}{2a_1\gamma_{\frac{m+1}{2}}}, \quad
{\rm{for}} \  \frac{i+1}{2} \ {\rm{odd}}.$$
$$
a_i=-\displaystyle\frac{a_1^2\gamma_{\frac{m-i+2}{2}}+2a_1a_3\gamma_{\frac{m-i+4}{2}}+\cdots
+(2a_1a_{i-2}+\cdots+2a_{\frac{i-5}{2}}a_{\frac{i+3}{2}}+a^2_{\frac{i-1}{2}})
\gamma_{\frac{m-1}{2}}}{2a_1\gamma_{\frac{m+1}{2}}}-$$
$$-\displaystyle\frac{(2a_3a_{i-2}+\dots+2a_{\frac{i-1}{2}}a_{\frac{i+3}{2}})
\gamma_{\frac{m+1}{2}}}{2a_1\gamma_{\frac{m+1}{2}}} , \quad
{\rm{for}} \ \frac{i+1}{2} \ {\rm{even}}.
$$ when $4\leq i \leq m,$
we obtain $[y'_m,y'_1]=x'_2,$ $[y'_i,y'_1]=0$ for $1\leq i\leq
m-1.$

Then applying Leibniz superidentity get the rest brackets
$$\begin{array}{ll}
[y'_i,y'_j]=0,&1\leq i,j\leq m,\ i+j\neq m+1, \\{}
[y'_{i},y'_j]=(-1)^{j-1}x'_2,&1\leq i,j\leq m,\ i+j=m+1.
\end{array}.$$ Thus, we obtain the superalgebra of the theorem. \end{proof}

{\bf{Case}} $\mathcal{ZF}^{n,m} \ (n\geq 3,\ m\geq 3).$

\

\begin{lem} Any Leibniz superalgebra from $\mathcal{ZF}^{n,m} \ (n\geq 3,\ m\geq 3)$ has
nilindex less that $n+m.$
\end{lem}
\begin{proof} Let us assume the contrary, i.e. $\ll$ is a Leibniz superalgebra
from $\mathcal{ZF}^{n,m} \ (n\geq 3,\ m\geq 3)$ and $\ll$ has the
nilindex equal to $n+m$. Then in the adapted basis we have
$$\ll=\{x_1,x_2,\dots,x_n,y_1,y_2,\dots,y_m\},$$
$$\ll^2=\{x_2,\dots,x_n,y_2,\dots,y_m\},$$
$$\ll^3\supset \{x_3,\dots,x_n,y_3,\dots,y_m\}.$$

Let us suppose that $\ll^3=\{x_3,\dots,x_n,y_2,\dots,y_m\},$ i.e.
$x_2\notin  \ll^3$ and $y_2\in \ll^3$. Then there exits $i_0$ ($2
\leq i_0 \leq n$) such that $[x_{i_0},y_1]=\alpha_{i_{0},2}
y_2+\cdots+\alpha_{i_{0},m} y_m$ with $\alpha_{i_{0},2}\neq 0.$
Since $x_i\in \mathcal{R}(\ll)$ for $2\leq i \leq n$ and
$\mathcal{R}(\ll)$ is an ideal, then  $\alpha_{i_{0},2}
y_2+\cdots+\alpha_{i-{0},m} y_m\in \mathcal{R}(\ll).$

Multiplying the product $[x_{i_0},y_1]$ on the right side
consequently to the basic element $x_1$ $(m-1)-$times we easily
obtain that $y_2, y_3,\dots, y_{m}\in \mathcal{R}(\ll),$ that is
$\ll^2=\mathcal{R}(\ll).$

By induction one can prove the following
$$[x_i,y_1]=\left\{\begin{array}{cc}
\sum\limits_{j=2}^{m+1-i}\alpha_{1,j}y_{j+i-1}, &  {\rm{if}} \quad
i+1 \leq m,\\[1mm]
0, & {\rm{if}} \quad i+1 > m.\\[1mm]
\end{array} \right.\eqno(4)$$
Since $\ll^2=\mathcal{R}(\ll)$ then $y_2$ can appear only in the
products $[x_i,y_1]$ for $2\leq i \leq n)$ or $[y_j,x_1]$ for
$2\leq j\leq m-1).$ However, in the first case from (4) we
conclude that $y_2$ does not lie in $\ll^3$ and in the second case
the element $y_2$ can not be obtained, i.e. in both cases we have
a contradiction with the assumption
$\ll^3=\{x_3,\dots,x_n,y_2,\dots,y_m\}.$

Thus, $\ll^3=\{x_2, \dots, x_n, y_3,\dots, y_m\}.$ Let $s$ be a
natural number such that $x_2\in \ll^s \setminus \ll^{s+1}.$

Suppose $s\leq m.$ Then we have
$$\ll^i=\{x_2,\dots,x_n,y_i,\dots,y_m\}, \quad 2 \leq i \leq s,$$
$$\ll^{s+1}=\{x_3,\dots,x_n,y_s,\dots,y_m\}$$
and in the equality
$[y_{s-1},y_1]=\sum\limits_{j=2}^n\beta_{s-1,j}x_j$ the
coefficient $\beta_{s-1,2}$ is not zero.

From Lemma \ref{family} we have
$$[y_1,y_s]=\displaystyle\sum_{i=0}^{s-1}(-1)^iC_{s-1}^i\sum_{t=2}^{n-s+i+1}
\beta_{1+i,t}x_{t+s-i-1},$$ in which the coefficient
$\beta_{s-1,2}$ occurs. Taking into account the equality
$[y_s,y_1]=\displaystyle\sum_{j=2}^{n}\beta_{s,j}x_j$ we conclude
that  $x_3\in lin<[y_1,y_s],[y_s,y_1],x_4,\dots,x_n>$. Therefore
$\ll^{s+2}=<x_3,\dots,x_n,y_{s+1},\dots,y_m>,$ i.e. $y_s\in
\ll^{s+1}\setminus \ll^{s+2}$ and $\alpha_{2,s}\neq 0.$

Consider the equalities
$$[y_{s-1},[y_1,y_1]]=2[[y_{s-1},y_1],y_1]=2\left[\sum_{t=2}^{n}\beta_{s-1,t}[x_t,y_1]\right]
=2\beta_{s-1,2}[x_2,y_1]+\sum_{i \geq s+1}(*)y_i.$$ On the other
hand $[y_{s-1},[y_1,y_1]]=0,$ because $[y_1,y_1]\in
\mathcal{R}(\ll).$

The basic element $y_s$ appears only in the product $[x_2,y_1].$
Hence we have that $\beta_{s-1,2}\alpha_{2,s}y_s+\sum\limits_{i
\geq s+1}(*)y_i=0$ which implies $\beta_{s-1,2}\alpha_{2,s}= 0.$
This contradicts to the assumption $s \leq m.$

Let us consider now the case $s=m+1.$ Then we have
$$\ll=\{x_1,x_2,\dots,x_n,y_1,y_2,\dots,y_m\},$$
$$\ll^i=\{x_2,x_3, \dots,x_n,y_i, y_{i+1}\dots,y_m\}, \quad 2 \leq i \leq m,$$
$$\ll^{m+i-1}=\{x_i,x_{i+1}\dots,x_n\}, \quad 2 \leq i \leq n.$$

Since $x_2\in\ll^{m+1}$ we have
$[y_m,y_1]=\sum\limits_{i=2}^n\beta_{m,i}x_i$ with $\beta_{m,2
}\neq 0.$

The sum $[y_1,x_1]+[x_1,y_1]$ lies in $\mathcal{R}(\ll)$ since
$[y_1,x_1]+[x_1,y_1]=(1+\alpha_{1,2})y_2+\alpha_{1,3}y_3+\cdots+\alpha_{1,m}y_m\in
\mathcal{R}(\ll)$.

If $[y_1,x_1]+[x_1,y_1]=0,$ then using the Leibniz superidentity
we have $$[x_1,[y_m,y_1]]=
[[x_1,y_m],y_1]+[[x_1,y_1],y_m]=-[y_2,y_m]=\beta_{m,2}x_3+\sum_{i\geq
4} (*)x_i.$$

On the other hand
$$[x_1,[y_m,y_1]]=\sum\limits_{i=2}^n\beta_{m,i}[x_1, x_i]=0.$$
Hence, $\beta_{m,2}=0$ which is a contradiction.

Thus, $[y_1,x_1]+[x_1,y_1]\neq 0.$ Continuing the same
argumentation as in the proof of Theorem \ref{thm32} we obtain
$y_m\in \mathcal{R}(\ll)$. Therefore
$$[y_1,y_m]=\displaystyle\sum_{i=0}^{m-1}(-1)^i C_{m-1}^i
\sum_{t=2}^{n-m+i+1}\beta_{1+i,t}x_{t+m-i-1}=0.$$ The minimal
value of the expression $t+m-i-1$ is reached when $i=m-1$ and
$t=2.$ Thus,
$[y_1,y_m]=(-1)^{m-1}C_{m-1}^{m-1}\beta_{m,2}x_2+\sum\limits_{i\geq
4}(*)x_i$ which implies that $\beta_{m,2}=0.$ That is a
contradiction with the assumption that nilindex of $\ll$ is equal
to $n+m.$
\end{proof}

\section{Leibniz superalgebras with the characteristic sequence $(n|m_1,m_2,\dots,m_k)$
and nilindex $n+m$}

Leibniz superalgebras with the characteristic sequence equal to
$(n|m-1,1)$ and with the nilindex $n+m$ were examined in
 \cite{C-G-N-O}. Therefore, in this section we shall consider the
Leibniz superalgebras $\ll$ of nilindex $n+m$ with the
characteristic sequence equal to $(n|m_1,m_2,\dots,m_k)$ with
conditions $m_1 \leq m-2$.

From the definition of characteristic sequence there exits a basis
$\{x_1, x_2, \dots x_n, y_1,\\ y_2, \dots y_m\}$ in which  the
operator $R_{x_1|_{L_1}}$ has the following form:

$$R_{{x}|_{\ll_1}}=\left(\begin{array}{llll}
J_{m_{j_1}}&0&\cdots&0\\
0&J_{m_{j_2}}&\cdots&0\\
\cdots&\cdots&\cdots&\cdots\\
0&0&\cdots&J_{m_{j_k}} \end{array}\right),$$ where $(m_{j_1},
m_{j_2},\dots, m_{j_k})$ is a permutation of $(m_1,
m_2,\dots,m_k).$ Without loss of generality, by a shifting of the
basic elements we can assume that operator $R_{x_1|_{L_1}}$ has
the following form
$$R_{x_1|_{\ll_1}}=\left(\begin{array}{llll}
J_{m_1}&0&\cdots&0\\
0&J_{m_2}&\cdots&0\\
\cdots&\cdots&\cdots&\cdots\\
0&0&\cdots&J_{m_k} \end{array}\right).$$ It means that the basis
$\{x_1, x_2, \dots x_n, y_1, y_2, \dots y_m\}$ satisfies the
following conditions:
$$\begin{array}{ll}
[x_i,x_1]=x_{i+1}, &  1 \leq i \leq n-1, \\{}
[y_j,x_1]=y_{j+1},&\mbox{for }j\notin
\{m_1,m_1+m_2,\dots,m_1+m_2+\cdots+m_k\},\\{}
[y_j,x_1]=0,&\mbox{for
}j\in\{m_1,m_1+m_2,\dots,m_1+m_2+\cdots+m_k\}.
\end{array}\eqno(5)$$

It is clear that two generators can not lie in $\ll_0.$ In fact,
in \cite{Omirov2} the table of multiplication of the Leibniz
algebra $\ll_0$ is presented and it has only one generator.

\begin{thm}\label{t2}
Let $\ll$ be a Leibniz superalgebra of nilindex $n+m$ with
characteristic sequence $(n|m_1,m_2,\dots,m_k),$ where $m_1\leq
m-2.$ Then both generators can not belong to $\ll_1$ at the same
time.
\end{thm}
\begin{proof} Let Leibniz superalgebra $\ll=\ll_0\oplus  \ll_1$ has nilindex $n+m$ and
let $\{x_1, x_2, \dots, x_n\}$ be a basis of $\ll_0$ and $\{y_1,
y_2, \dots, y_m\}$ a basis of $\ll_1$. Suppose that two generators
lie in $\ll_1.$ Then they should be from the set
$$\{y_1, y_{m_1+1}, y_{m_1+m_2+1}, \dots, y_{m_1+m_2+\cdots+m_{k-2}+1},
y_{m_1+m_2+\cdots+m_{k-1}+1}\}$$

Without loss of generality, the generators can be chosen as
$\{y_1, y_{m_1+1}\}.$

Consider the following cases:

\noindent {\bf Case 1.} Let $[y_1,y_1]\in \ll_0\setminus \ll_0^2.$
Then consider the Leibniz superalgebra generated by the element
$<y_1>.$ Since $[y_1,y_1]\in \ll_0\setminus\ll_0^2$ we can assume
$[y_1,y_1]=x_1.$ Then from the products in (5) we deduce $\{x_1,
x_2, \dots, x_n, y_2, y_3, \dots, y_{m_1}\}\subseteq <y_1>.$ It it
easy to see that $y_{m_1+1}\notin <y_1>.$ Indeed, if $y_{m_1+1}\in
<y_1>,$ then $\{y_{m_1+2},\dots,y_{m_1+m_2}\} \subseteq <y_1>$
which implies $C(\ll)\geq (n \ | \ m_1+m_2, m_3, \dots, m_k).$
That is a contradiction to the condition of characteristic
sequence of $\ll$, because $C(\ll)= (n|m_1, m_2, \dots, m_k).$

Thus, the Leibniz superalgebra generated by the basic element
$y_1$ consist of
$$\{x_1,x_2,\dots,x_n,y_1,y_2,\dots,y_{m_1}\}.$$
Since the superalgebra $<y_1>$ is single-generated then from
Theorem \ref{t1} we have that either $m_1=n$ or $m_1=n+1$ and the
multiplication in $<y_1>$ has the following form:
$$\begin{array}{ll}
[x_i,x_1]=x_{i+1},&1\leq i\leq n-1,\\{} [y_j,x_1]=y_{j+1},&1\leq
j\leq m_1-1,\\{} [x_i,y_1]=\frac{1}{2}y_{i+1},&1\leq i\leq
m_1-1,\\{} [y_j,y_1]=x_j,&1\leq j\leq n.
\end{array}$$

\

{\bf Case $m_1=n$}. Since $y_1$ and $y_{n+1}$ are generators we
have
$$\ll=\{x_1,x_2,\dots,x_n,y_1,\dots,y_n,y_{n+1},\dots,y_m\},$$
$$\ll^2=\{x_1,x_2,\dots,x_n,y_2,\dots,y_n,y_{n+2},\dots,y_m\}.$$
Besides, $x_1 \notin \ll^3.$ Otherwise, if $x_1 \in \ll^3,$ then
there exists $z\in \ll_1$ such that $z\in \ll^2/\ll^3.$ Thereby
$z\in lin <
[y_1,y_1],[y_1,y_{n+1}],[y_{n+1},y_1],[y_{n+1},y_{n+1}]>$ and
taking into account that $[y_i,y_j]\in \ll_0$ we obtain $z \in
\ll_0$ which is a contradiction. Thus,
$$\ll^3=\{x_2,\dots,x_n,y_2,\dots,y_n,y_{n+2},\dots,y_m\}.$$

If $\ll^{2k}=\{x_i,x_{i+1},\dots,x_n,y_j,$
$\dots,y_n,y_{n+2},\dots,y_m\}$, then by a similar way one can
prove that
$\ll^{2k+1}=\{x_{i+1},\dots,x_n,y_j,\dots,y_n,y_{n+2},\dots,y_m\}$.
In fact, if $z\in \ll^{2k}/\ll^{2k+1},$ then $z$ has to be
generated by $2k$ products of the generators (but they are from
$\ll_1$). Hence this products belong to $\ll_0,$ and we have $z\in
\ll_0.$

Applying the similar argumentation we get
$$\ll^{2k+2}=\{x_{i+1},\dots,x_n,y_{j+1},\dots,y_n,y_{n+2},\dots,y_m\}.$$

Continuing with the process, we obtain that
$\ll^{2n+1}=\{y_{i_1},y_{i_2},\dots,y_{i_k}\}$ and $\ll^{2n+2}=0.$
Since $dim(\ll^{2n+1}/\ll^{2n+2})=1$ then $\ll^{2n+1}=\{y_{n+2}\}$
and nilindex should be equal to $2n+2.$ Thus, $m=n+2$ and we have
$$\ll=\{x_1,x_2,\dots,x_n,y_1,\dots,y_n,y_{n+1},y_{n+2}\},$$
$$\ll^{2k}=\{x_k,\dots,x_n,y_{k+1},\dots,y_n,y_{n+2}\},\ \ 1\leq k\leq n-1,$$
$$\ll^{2k+1}=\{x_{k+1},\dots,x_n,y_{k+1},\dots,y_n,y_{n+2}\},\ \ 1\leq k\leq n-1,$$
$$\ll^{2n}=\{x_n,y_{n+2}\}, \ \ \ll^{2n+1}=\{y_{n+2}\}, \ \ \ll^{2n+2}=\{0\}.$$

Furthermore, $\ll^{2n}=[\ll^{2n-1},\ll]=<[x_n,y_1],$
$[x_n,y_{n+1}],$ $[y_n,y_1],$ $[y_n,y_{n+1}],$ $[y_{n+2},y_1],$
$[y_{n+2},y_{n+1}]>.$ Note that the element $y_{n+2}$ can be
obtained only from product $[x_n,y_{n+1}]$ (because $[x_n,y_1]=0$,
otherwise we get a contradiction with the property of
characteristic sequence). However,
$$[x_n,y_{n+1}]=[[x_{n-1},x_1],y_{n+1}]=[x_{n-1},[x_1,y_{n+1}]]+[[x_{n-1},y_{n+1}],x_1]=0,$$
which deduce $\ll^{2n}=\{x_n\},$ that is a contradiction with the
condition of nilindex.

\

{\bf Case $m_1=n+1$}. In this case, similar to the previous case,
we get a contradiction.

\noindent {\bf Case 2.} Let $[y_1,y_1]\notin \ll_0\setminus
\ll_0^2$ and $[y_{m_1+1},y_{m_1+1}]\in\ll_0 \setminus \ll_0^2.$
Then applying the same arguments for $y_{m_1+1}$ as for $y_1$ in
Case 1, we obtain a contradiction with the fact that both
generators lie in $\ll_1,$ as well.

\

\noindent {\bf Case 3.} Let $[y_1,y_1]\notin \ll_0 \setminus
\ll_0^2$ and $[y_{m_1+1},y_{m_1+1}]\notin \ll_0 \setminus
\ll_0^2.$ Then, without loss of generality, we can assume that
$$\begin{array}{l}
[y_1,y_{m_1+1}]=x_1,\\{}
[y_{m_1+1},y_1]=\sum\limits_{i=1}^{n}b_ix_i.
\end{array}$$

If $b_1=1,$ then making the change of basis $y'_1=y_1+y_{m_1+1}$
we obtain $[y'_1,y'_1]\in \ll_0 \setminus \ll_0^2.$ Therefore this
case can be reduced to Case 1.

If $b_1\neq 1,$ then
$[y_1,y_{m_1+1}]-[y_{m_1+1},y_1]=(1-b_1)x_1+b_2x_2+\cdots+b_nx_n\in
\mathcal{R}(\ll)$ and since $x_i\in \mathcal{R}(\ll)$ ($2\leq i
\leq n$) we get $x_1\in \mathcal{R}(\ll)$. From the Leibniz
superidentity we obtain $\ll^3=\{0\}$ and therefore $n=1, \ m=2.$
Obviously, we get a Leibniz algebra, i.e. the Leibniz superalgebra
with condition $m=0.$
\end{proof}

In other words, Theorem \ref{t2} claims that one generator lies in
$\ll_0$ and another one belongs to $\ll_1.$ Evidently, $x_1$ is a
generator and as a generator in $\ll_1$ we can chose $y_1.$

Put $$[x_i,y_1]= \sum\limits_{t=2}^m \alpha_{i,t}y_t,\quad 1 \le i
\le n,$$
$$[y_j,y_1]=
\sum\limits_{s=2}^n \beta_{j,s}x_s, \quad 1 \le j \le m.$$

The following equality can be proved by induction
$$[y_i,y_j]=\displaystyle\sum_{s=0}^{min\{i+j-1,m_1\}-i}(-1)^s C_{j-1}^s
\sum_{t=2}^{n-j+s+1}\beta_{i+s,t}x_{t+j-s-1}\eqno(6)$$ where
$1\leq i, j\leq m_1.$

\begin{thm} Let $\ll$ be a Leibniz superalgebra with characteristic sequence equal to $(n|m_1,
m_2, \dots, m_k),$ where $m_1 \leq m-2$. Then nilindex of $\ll$ is
less than $n+m.$
\end{thm}
\begin{proof} Let us suppose the contrary, i.e. the nilindex of $\ll$ is equal
to $n+m.$ Then $dim \ll^{k}=n+m-k, \ 2 \leq k \leq n+m.$ In the
adapted basis $\{x_1, x_2, \dots, x_n, y_1, y_2, \dots, y_m\}$ we
have the products
$$\begin{array}{ll}
[x_i,x_1]=x_{i+1},&1\leq i\leq n-1,\\{}
[y_j,x_1]=y_{j+1},&j\notin\{m_1,m_1+m_2,\dots,m_1+m_2+\cdots+m_k\},\\{}
[y_j,y_1]=\beta_{j,2}x_2+\cdots+\beta_{j,n}x_n,&1\leq j\leq m,\\{}
[x_i,y_1]=\alpha_{i,2}y_2+\cdots+\alpha_{i,m}y_m,&1\leq i\leq n.
\end{array}$$

Suppose that $x_2\in \ll^3,$ i.e.
$\ll^3=\{x_2,\dots,x_n,y_3,\dots,y_m\}.$ Then the element $x_2$ is
generated from the products $[y_j,y_1]$, i.e. there exists $j_0 \
(2 \leq j_0 \leq m)$ such that in
$[y_{j_0},y_1]=\sum\limits_{i=2}^n \beta_{j_0,i}x_i$ the parameter
$\beta_{j_0,2}\neq 0.$

Taking the change of basis
$x'_1=\frac{1}{\beta_{j_0,2}}(\sum\limits_{s=2}^n
\beta_{j_0,s}x_{s-1})$ we can assume that $[y_{j_0},y_1]=x_2.$

The equalities
$$[y_{j_0},[y_1,y_1]]=2[[y_{j_0},y_1],y_1]=2[x_2,y_1].$$
and
$$[y_{j_0},[y_1,y_1]]=[y_{j_0},\sum\limits_{i=2}^n
\beta_{1,i}x_i]=0$$ imply $[x_2,y_1]=0.$

Since $\ll^3 = n+m-3$ we have  $\ll^3=\{x_2,$ $x_3,\dots,$ $x_n,$
$y_3,\dots,$ $y_{m_1},$
$A_{1,1}y_2+A_{1,2}y_{m_1+1}+\cdots+A_{1,k}y_{m_1+m_2+\cdots+m_{k-1}+1},$
$y_{m_1+2},\dots,$ $y_{m_1+m_2},$
$A_{2,1}y_2+A_{2,2}y_{m_1+1}+\cdots+A_{2,k}y_{m_1+m_2+\cdots+m_{k-1}+1},
\dots,$ $ y_{m_1+m_2+\cdots+m_{k-1}},\dots,$
$A_{k-1,1}y_2+A_{k-1,2}y_3+\cdots+A_{k-1,k}y_{m_1+m_2+\cdots+m_{k-1}+1},\dots,$
$y_m\}$

If there exists $i$ such that
$[x_i,y_1]=c_2(A_{i,1}y_2+A_{i,2}y_{m_1+1}+\cdots+A_{i,k}y_{m_1+\cdots+m_{k-1}+1})+
\sum\limits_{s\geq3}(*)y_s$ with $c_2\neq 0,$ then applying the
Leibniz superidentity for the elements $\{x_i, x_1, y_1\}$
inductively, we obtain $A_{i,1}=0$ for all $i \ (1\leq i \leq
k-1)$.

Thus, $$\ll^3=\{x_2,x_3,\dots,x_n,y_3,\dots,y_{m-1},y_m\}.$$

Consider the product
$[x_1,y_1]=\sum\limits_{i=2}^n\alpha_{1,i}y_i.$

\

\textbf{Case 1.} Let $\alpha_{1,2}\neq 0$. Let us suppose that
there exists $s$ from $3\leq s\leq m_1$ such that
$$\ll^s=\{x_2,\dots,x_n,y_s,\dots,y_m\},$$
$$\ll^{s+1}=\{x_3,\dots,x_n,y_s,\dots,y_m\}.$$
Note that $\beta_{s-1,2}\neq 0,$ because $x_2\in \ll^s \setminus
\ll^{s+1}.$ The equality $[x_2,y_1]=0$ implies that the element
$y_s$ must be generated by the products $[x_i,y_1]$ for $3 \leq i
\leq n.$ Therefore,
$$\ll^{s+2}=\{x_4,\dots,x_n,y_s,\dots,y_{m-1},y_m\}.$$

The parameters  $\beta_{s,2}, \ \beta_{s,3}=0$ are equal to zero
because of $y_s \in \ll^{s+2}.$ Thus, we have
$$[y_{s-1},y_1]=\sum\limits_{i=2}^n
\beta_{s-1,i}x_i \ \ \mbox{and} \ \ [y_s,y_1]=\sum\limits_{i=4}^n
\beta_{s,i}x_i.$$

From the equality (6) we get
$$[y_2,y_{s-1}]=(-1)^{s-1}\beta_{s-1,2}x_3+\sum\limits_{t\geq 4}(*)x_t.$$

The following chain of the equalities
$$0=[x_1, [y_1, y_{s-1}]]= [[x_1, y_1], y_{s-1}] +[[x_1, y_{s-1}],
y_1]=$$ $$[\sum\limits_{i=2}^n\alpha_{1,i}y_i,
y_{s-1}]+[\sum\limits_{i=s}^n\gamma_{1,i}y_i, y_1]= \alpha_{1,2}
\beta_{s-1,2}x_3+\sum\limits_{t\geq 4}(*)x_t$$ implies
$\alpha_{1,2} \beta_{s-1,2}=0$, that is a contradiction with the
assumption $s\leq m_1.$

Thus, we have $s>m_1.$ Then
$$\ll^{m_1}=\{x_2,\dots,x_n,y_{m_1},\dots,y_m\},$$
$$\ll^{m_1+1}=\{x_2,\dots,x_n,y_{m_1+1},\dots,y_m\}.$$
Since $\ll^{m_1+2}=[\ll^{m_1+1}, \ll]$ then using the equality (6)
we conclude that
$$\ll^{m_1+2}=\{x_3,\dots,x_n,y_{m_1+1},\dots,y_m\}.$$
Therefore, $s=m_1+1$ and
$[y_{m_1},y_1]=\sum\limits_{i-2}^n\beta_{m_1,2}x_2$ with
$\beta_{m_1,2}\neq 0$. From (6) we get
$[y_2,y_{m_1}]=\beta_{m_1,2}x_3+\sum\limits_{i\geq 4}x_i$ and
using the Leibniz superidentity we obtain $\beta_{m_1,2}=0,$ but
it is a contradiction. Hence this case is not possible.

\

\textbf{Case 2.} Let $\alpha_{1,2}=0.$ Then we have that
$$\begin{array}{l}
[x_1,y_1]=\alpha_{1,3}y_3+\cdots+\alpha_{1,m}y_m,\\{}
[y_1,x_1]=y_2,\\{}
[y_1,x_1]+[x_1,y_1]=y_2+\alpha_{1,3}y_3+\cdots+\alpha_{1,m}y_m\in
\mathcal{R}(\ll).
\end{array}$$

By the similar arguments as in the previous case, we obtain
$y_i\in \mathcal{R}(\ll) \ (2 \leq i \leq m).$ Applying the
Leibniz superidentity for the elements $\{y_{j-1}, x_1, y_1\}$  we
have

$$[y_j,y_1]=\sum\limits_{i=2}^{n+1-j}\beta_{j,i}x_{i+j-1}, \ 2 \leq j \leq m_1$$
which imlpies $\beta_{i,2}=0$ for all $2 \leq i \leq m_1.$

Since $y_{m_1+\dots+m_{s-1}+1}$ ($2 \leq s \leq k$) is generated
from the products $[x_t,y_1]$ ($1 \leq t \leq n$), then from the
equality
$$[[x_t,y_1],y_1]=\frac{1}{2}[x_t,[y_1,y_1]]=[x_t, \sum\limits_{l\geq
2}(*)x_l]=0
$$
and applying the Leibniz superidentity for the elements
$\{y_{j-1}, x_1, y_1\}$ we get $\beta_{i,2}=0$ for $2 \leq i \leq
m,$ but it is a contradiction to the assumption $x_2 \in \ll^3.$

Thus, we have $x_2\notin \ll^3,$ i.e.
$$\ll^3=\{x_3, x_4, \dots,x_n, y_2,y_3,\dots,y_m\}.$$
Since $\{y_2, y_{m_1}, y_{m_1+m_2}, \dots,
y_{m_1+m_2+\cdots+m_k}\}\in\ll^3,$ then there exist
$i_1,i_2,\dots,i_k\geq 2,$ such that
$$ \left |
\begin{array}{lllll}
\alpha_{i_1,2}&\alpha_{i_1,m_1+1}&\alpha_{i_1,m_1+m_2+1}&\cdots&\alpha_{i_1,m_1+m_2+\cdots+m_{k-1}+1}\\
\alpha_{i_2,2}&\alpha_{i_2,m_1+1}&\alpha_{i_2,m_1+m_2+1}&\cdots&\alpha_{i_2,m_1+m_2+\cdots+m_{k-1}+1}\\
\vdots&\vdots&\vdots& &\vdots\\{}
\alpha_{i_k,2}&\alpha_{i_k,m_1+1}&\alpha_{i_k,m_1+m_2+1}&\cdots&\alpha_{i_k,m_1+m_2+\cdots+m_{k-1}+1}
\end{array}\right |\neq 0 . \eqno(7)$$

Without loss of generality we can assume that $\alpha_{i_1,2}\neq
0$. Consider
$$[x_{i_1},y_1]+[y_1,x_{i_1}]=\sum\limits_{t=2}^m\alpha_{i_1,t} y_t \in
\mathcal{R}(L).$$ Then multiplying sufficiently times from the
right side to the element $x_1$ and taking into account the
condition (7) we obtain $y_i\in \mathcal{R}(\ll) \ (2\leq i\leq
m).$

Furthermore, proceeding with the bracket computing
$$\begin{array}{l}
[x_2,y_1]=[[x_1,y_1],x_1]=\alpha_{1,2}
y_3+\cdots+\alpha_{1,m-1}y_m,\\{}
[x_3,y_1]=[[x_2,y_1],x_1]=\alpha_{1,2}
y_4+\cdots+\alpha_{1,m-2}y_m,\\{} \qquad \qquad \qquad \vdots\\{}
[x_n,y_1]=[[x_{n-1},y_1],x_1]=\alpha_{1,2}
y_{n+1}+\cdots+\alpha_{1,n-m+1}y_m.
\end{array}$$
we obtain that $y_2\notin \ll^3.$

Thus, we get the contradictions in all considered cases, which
leads that the superalgebra $\ll$ with characteristic sequence
$C(\ll)=(n|m_1, m_2, \dots, m_k)$ and $m_1 \leq m-2$ has a
nilindex less than $n+m.$
\end{proof}
Combining the assertion of Theorem \ref{thm32} and the
classifications of the papers \cite{C-G-N-O}, \cite{G-N-O} we
complete the classification of Leibniz superalgebras with even
part zero-filiform Leibniz algebra and with nilindex equal to
$n+m.$


\begin{thebibliography}{12}
\bibitem{Alb} Albeverio S., Ayupov Sh.A., Omirov B.A.
{\it On nilpotent and simple Leibniz algebras.} Comm. in Algebra,
19(1), (2005), p. 14--25.

\bibitem{Omirov2} Ayupov Sh.A., Omirov B.A. {\it On some classes of nilpotent
Leibniz algebras.} Siberian Math. J., 42(1), (2001), p. 18--29.

\bibitem{Alb-Omirov-Rak} Albeverio S., Omirov
B.A., Rakhimov I.S. {\it Varieties of nilpotent complex Leibniz
algebras of dimension less than five}. Comm. in Algebra, 33(5),
(2005), p. 1575--1585.

\bibitem{FilSup} Ayupov Sh. A., Khudoyberdiyev A. Kh., Omirov B. A. {\it The
classification of filiform Leibniz superalgebras of nilindex n+m},
appear in Acta Math.Sinica (english series) 2009,
arXiv:math/0611640.

\bibitem{2007Yu} Bordemann M., G\'{o}mez J.R., Khakimdjanov Yu., Navarro R.M.
{\it Some deformations of nilpotent Lie superalgebras}, J. Geom.
and Phys., 57, (2007), p. 1391--1403

\bibitem{C-G-N-O} Camacho L.M., G\'{o}mez J.R., Navarro R.M., Omirov B.A.
{\it Classification of some nilpotent class of Leibniz
superalgebras}, submitted to Acta Math.Sinica (english series),
arXiv:math/0611636.


\bibitem{GL} Gilg M. {\it Super-alg\`{e}bres de Lie nilpotentes:} PhD thesis. University of Haute Alsace,
2000. -- 126 p.

\bibitem{GL1} Gilg M. {\it On Deformations of the Filiform Lie Superalgebra $L^{n,m}$},
Comm. in Algebra, 32(6), (2004), p. 2099--2115.

\bibitem{G-N-O} G\'{o}mez J.R., Navarro R.M., Omirov B.A.
{\it On nilponent Leibniz Superalgebras}, arXiv:math/0611723,
submitted to J. of Algebra.

\bibitem{G-K-N} G\'{o}mez J.R., Khakimjanov Yu., Navarro R.M.
{\it Some problems concerning to nilpotent Lie superalgebras}, J.
Geom. and Phys., 51(4), (2004), p. 473--486.

\bibitem{G-K-N1} G\'{o}mez J.R., Khakimjanov Yu., Navarro R.M.
{\it Infinitesimal deformations of the Lie superalgebra
$L^{n,m}$}, J. Geom. and Phys., 58, (2008), p. 849--859.


\bibitem{K} Kac V. G. {\it Lie superalgebras}, Advances in Math., 26(1), (1977),
p. 8--96.

\bibitem{loday} Loday J.-L. {\it Une version non commutative des
alg$\acute{e}$bres de Lie: les alg$\acute{e}$bres de Leibniz},
Ens. Math., 39, (1993), p. 269--293.

\end{thebibliography}
\end{document}